\documentclass{amsart}
\usepackage[nobysame,initials]{amsrefs}
\usepackage{graphicx}
\usepackage{xcolor}

\newtheorem{theorem}{Theorem}
\newtheorem{lemma}[theorem]{Lemma}
\newtheorem{cor}[theorem]{Corollary}

\newcommand{\EE}{\mathbb{E}}
\newcommand{\EEE}{\mathbf{E}}
\newcommand{\dx}{{\rm d}}

\newcommand{\dd}{\mathrm{d}}

\DeclareMathOperator{\area}{Area}

\title{On random disc-polygons in a disc-polygon}

\author{Ferenc Fodor} 
\address{Department of Geometry, Bolyai Institute, University of Szeged, 
Aradi v\'ertan\'uk tere 1, H-6720 Szeged, Hungary} 
\email{fodorf@math.u-szeged.hu}

\author{P\'eter Kevei}
\address{ 
Department of Stochastics, Bolyai Institute, University of Szeged, Aradi v\'ertan\'uk tere 1, H-6720 Szeged, Hungary}
\email{kevei@math.u-szeged.hu}

\author{Viktor V\'{\i}gh}
\address{Department of Geometry, Bolyai Institute, 
University of Szeged, Aradi v\'ertan\'uk tere 1, H-6720 Szeged, Hungary}
\email{vigvik@math.u-szeged.hu}

\begin{document}
\maketitle

\begin{abstract}
We prove asymptotic formulas for the expectation of the vertex number and missed area of uniform random disc-polygons in convex disc-polygons. Our statements are the $r$-convex analogues of the classical results of R\'enyi and Sulanke \cite{RS64} about random polygons in convex polygons.  
\end{abstract}

\section{Introduction and results}
Let $K$ be a convex body (compact convex set with non-empty interior) in $d$-dimensional Euclidean space $\EEE^d$, and let $X_n=\{x_1,\ldots, x_n\}$ be independent random points from $K$ chosen according to the uniform probability distribution (the Lebesgue measure in $K$ normalised by the volume of $K$). The convex hull $K_n^*=[X_n]$ of $X_n$ is a (uniform) random polytope in $K$. The behaviour of the geometric properties of $K_n^*$ have been investigated extensively. In particular, the study of the asymptotic properties of $K_n^*$ started when, in the plane, R\'enyi and Sulanke \cites{RS63, RS64} determined the behaviour of the expectations of the vertex number of $K_n^*$ and the $\area(K\setminus K_n^*)$ missed by $K_n^*$, as $n\to\infty$ in the case when $K$ is convex polygon or a sufficiently smooth disc. For a detailed overview of known results about this classical model we refer to the surveys by B\'ar\'any \cite{B08}, Reitzner \cite{R03},  Schneider \cite{Sch17}, and the references therein.

In this paper we work in the Euclidean plane $\EEE^2$ and consider a modification of the classical probability model of random polygons in which we use intersections of congruent circles to generate an analogue of the classical convex hull.  

Let $B$ denote the origin centred unit ball of $\EEE^2$, and let $S^1=\partial B$ be its boundary. 
For a fixed $r>0$, an $r$-disc-polygon is a compact convex set in $\EEE^2$ that is bounded by a finite number of radius $r$ circular arcs. Let $X\subset \EEE^2$ be a finite point set that is contained in a closed circle of radius $r$. The intersection of all radius $r$ closed circular discs that contains $X$, denoted by $[X]_r$, is an $r$-disc-polygon. The vertices and edges of a disc-polygon are defined in the natural way. It is known, see, for example, \cite{BLNP07} that if $P$ is an $r$-disc-polygon and $X\subset P$
, then $[X]_r\subset P$. Furthermore, 
for each boundary point $x\in\partial P$, there exists a point $v\in\EEE^2$ such that $x\in rS^1+v$ and $P\subset rB+v$. We call such $rB+v$ a supporting disc of $P$. Note that if $x$ is a vertex of $P$, then there are infinitely many vectors $v$ with this property, therefore, in this case the supporting disc is not unique.

Let $P$ be an $r$-disc-polygon in $\EEE^2$, and let $X_n=\{x_1, x_2, \ldots, x_n\}$ be a sample of $n$ independent 
random points in $P$ chosen according to the uniform probability distribution. 
The closed $r$-hull $P_n^r=[X_n]_r$ is a {\em uniform random $r$-disc-polygon} in $P$. 

Let $f_0 (\cdot)$ be the number of vertices of a convex (disc-)polygon, and let $\area(\cdot)$. 
In \cite{RS64} R\'enyi and Sulanke proved that if $P$ is a (classical) convex polygon, then 
\begin{align}\label{RS-polygon}
\lim_{n\to\infty}\frac{\EE f_0 (P_n^*)}{\ln n}&=\frac 23f_0 (P).
\end{align}
In fact, their formula is more precise than \eqref{RS-polygon} but we state it here in this simpler form as it fits the following discussion better.
It is a natural question: what is the asymptotics of $\EE f_0 (P_n^r)$ if $P$ is a $r$-disc-polygon? 
Our main result is the following theorem that answers this question:
\begin{theorem}\label{main}
If $P$ is a convex $r$-disc-polygon, then
\begin{equation}\label{main:vert}
\lim_{n\to \infty} \frac{\EE f_0 (P_n^r)}{\ln n}=\frac{2}{3} f_0(P)
\end{equation}
and
\begin{equation}\label{main:area}
\lim_{n\to \infty} \frac{n\EE \area(P\setminus P_n^r)}{\ln n}=\frac{2}{3}f_0(P)\area(P).
\end{equation}
\end{theorem}

The quantity $\area (P\setminus P_n^r)$
is often called the missed area of $P$, and the 
limit formula \eqref{main:area} follows from \eqref{main:vert} by the $r$-convex analogue of Efron's identity, cf. \cite{FKV14}. Subsequently, we will prove \eqref{main:vert} in detail.

We would like to point out that our argument is very different from the one used by R\'enyi and Sulanke in \cite{RS64}, where affine invariance played a key role in the proof of \eqref{RS-polygon}. This is not an option in our case as the model is not invariant under affine transformations. Therefore, in order to evaluate \eqref{dabjuen}, one needs to use techniques that are more essentially based on the geometric properties of the model. This extra geometric information is described in Section 2 and it mainly concerns the behaviour of small disc-caps which determines how to divide the domain of integration in \eqref{dabjuen}. 

It is a natural question to ask how Theorem~\ref{main} is related to the corresponding classical result \eqref{RS-polygon} of R\'enyi and Sulanke \cite{RS63}. 
Our method can also be used, with some modifications, to prove \eqref{RS-polygon}. However, whether \eqref{main:vert} implies \eqref{RS-polygon} in the limit as $r\to\infty$ is unclear.

We call a compact convex set $K\subset {\bf E}^2$ $R$-convex (the terms $R$-spindle convex and $R$-hyperconvex are also used in the literature), if it is the intersection of all radius $R$ closed circular discs that contain $K$. This condition is known to be equivalent to the property that $K$ slides freely in a circle of radius $R$, that is, for any $x\in RS^1$ there exists a vector $p\in{\mathbb E}^2$ with $x\in K+p\subset RB$. The concept of $R$-convexity goes back, at least, to Mayer \cite{M1935}, and it has been investigated recently quite intensively. The importance of $R$-convexity comes, in part, from its connection to various old problems in which intersections of congruent balls appear, like the Kneser-Poulsen conjecture. For more information on the properties of $R$-convex sets we refer to \cite{BLNP07}, \cite{FKV14} and the references therein.

Our probability model has a natural modification for $R$-convex discs. If $K$ is an $R$-convex disc for some $R\leq r$, and $X_n=\{x_1, \ldots, x_n\}$ are independent random points chosen from $K$ according to the uniform probability distribution, then it is known that the random $r$-disc-polygon  $K_n^r=[X_n]_r$ is contained in $K$, see \cite{FKV14}. 
The asymptotic behaviour of the expectations $\EE f_0 (K_n^r)$ and $\EE \area(K\setminus K_n^r)$ have been determined by Fodor, Kevei and V\'\i gh in \cite{FKV14} in the case when $K$ is a convex disc such that it boundary $\partial K$ is $C^2_+$ and $r>1/\kappa_m$, where $\kappa_m=\min_{x\in\partial K}\kappa (x)>0$ and $\kappa (x)$ denotes the curvature of $\partial K$ at $x$. It is known that under these conditions $K$ is $R$-convex for $R\ge 1/\kappa_m$, see \cite{Sch14}*{Theorem~3.2.12 on p. 164}. The following statements were proved in \cite{FKV14}:

\begin{align}
\lim_{n\to\infty}\EE f_0(K_n^r)\cdot n^{1/3}&=\sqrt[3]{\frac{2}{3\area (K)}}\Gamma\left (\frac 53\right )c(K,r),\label{vertex-asympt}\\
\lim_{n\to\infty}\EE \area(K\setminus K_n^r)\cdot n^{2/3}&=\sqrt[3]{\frac{2\area(K)^2}{3}}\Gamma\left (\frac 53\right )c(K,r),\label{missed-area-asympt}
\end{align}
where 
$$c(K,r)=\int_{\partial K}\left (\kappa (x)-\frac 1r\right )^{1/3}\, \dx x.$$
The symbol $\Gamma(\cdot)$ denotes Euler's gamma function, and integration on $\partial K$ is with respect to arc-length.

The formulas \eqref{vertex-asympt} and \eqref{missed-area-asympt} are generalisations of the corresponding classical results of R\'enyi and Sulanke from \cite{RS63} in the sense that the asymptotic formulas of R\'enyi and Sulanke follow from \eqref{vertex-asympt} and \eqref{missed-area-asympt} in the limit as $r\to\infty$, see Section~3 of \cite{FKV14} for details.

Finally, we conjecture that for any $r$-convex disc $K\subset \EEE^2$ different from $rB^2$ the following inequalities hold for any $n$
\begin{align}\label{ineq:lower-upper}
c_1(K)\log n<\EE f_0(K_n^r)<c_2(K)n^{1/3},
\end{align}
for suitable constants $c_1(K)$ and $c_2(K)$, and that the orders in \eqref{ineq:lower-upper} are optimal: the left-hand inequality of is realised by $r$-disc-polygons and the right-hand inequality by smooth $r$-convex discs. 
We note that, due to the different behaviour of $rB^2$, it has to be excluded from the inequality \eqref{ineq:lower-upper}, cf.~Theorem~1.3 in \cite{FKV14}.  

The corresponding inequalities in the classical convex case for the number $f_k(\cdot)$ of $k$-dimensional faces were established using floating bodies and the Economic Cap Covering Theorem by B\'ar\'any and Larman \cite{BL88} and by B\'ar\'any \cite{B89}: for any convex body $K \subset \EEE^d$ it holds that
\begin{align}
C_1(d)(\log n)^{d-1}<\EE f_k(K_n^*)<C_2(d)n^{\frac{d-1}{d+1}}
\end{align}
for suitable constants $C_1(d)$ and $C_2(d)$ and any $n$. Here the left-hand inequality is of right order for polytopes and the right-hand one for smooth convex bodies. 

Unfortunately, the analogue of the Economic Cap Covering Theorem is not known for the $r$-convex case, even in the plane. We conjecture that it is true, however, the methods used in its proof do not seem to translate to the $r$-convex setting.

\section{Caps of disc-polygons}
As both \eqref{main:vert} and \eqref{main:area} are invariant under simultaneous scaling of $K$ and the generating circles of $P_n^r$, we may and do assume from now on that $r=1$ and omit $r$ from the notation. Accordingly, we use the $[X]_S$ symbol for the $1$-hull of the set $X$.
In particular, the $1$-hull of two points $x,y\in\EEE^2$, with $|x-y|\leq 2$ is denoted by $[x,y]_S$ and is called the {\em spindle} of $x$ and $y$. Subsequently, a disc-polygon always means a convex $1$-disc-polygon.   

Let $P$ be a disc-polygon and let $B^\circ$ denote the origin centred unit radius open circular disc.
A subset $D$ of $P$ is a {\em disc-cap of $P$} if 
$D=P\setminus (B^\circ+p)$ for some point $p\in\EEE^2$.
Note that in this case $\partial B+p$ intersects 
$\partial P$ in at most two points, and $D$ contains at least one vertex of $P$. The
boundary of a nonempty disc-cap $D$ consists
of at most two connected arcs: one arc is a subset of $\partial P$, and
the other arc is a subset of $\partial B+p$.   

For $x\in \partial P$, let $\mathcal N(x)\subset S^1$ denote the set of all outer unit normal vectors of $P$ at $x$. If $x\in\partial P$ is not a vertex of $P$, then $\mathcal N(x)=\{u_x\}$ contains a single element. If $x$ is a vertex, then $\mathcal N(x)$ determines a closed and connected arc of $S^1$. 

\begin{lemma}\label{vertexclaim}
Let $P$ be a disc-polygon.
Let $D=P\setminus (B^\circ +p)$ be a non-empty disc-cap of $P$ with non-empty interior.  Then there exists a unique
unit vector $u$ and a number $t> 0$ such that $B+p=B+x_0-(1+t) u$, where $x_0$ in the unique point on $\partial P$ with $u\in \mathcal N (x_0)$.
\end{lemma}
We call $u$ the outer unit {\em normal}, $x_0$ the {\em vertex}, and $t$ the {\em height of $D$}.
Lemma~\ref{vertexclaim} was proved in \cite{FKV14} for the $C^2_+$ case, and in higher dimension in \cite{F19} also for the $C^1$ case. Essentially the same argument works here too but for the sake of completeness we provide a short proof.

\begin{proof}
Let $x_0$ be a point of $P$ whose distance from $p$ is maximal. First we show that $x_0$ is unique. Assume on the contrary that $x_1\neq x_0$ are both at maximal distance from $p$. Then the spindle $[x_0,x_1]_S$ is also in $P$, and one of the midpoints of the unit circular arcs connecting $x_0$ and $x_1$ is farther from $p$ than $x_0$, a contradiction.

Let $u=(x_0-p)/|x_0-p|\in S^1$. The line through $x_0$ that is perpendicular to $u$ clearly supports $P$ at $x_0$ hence $u\in \mathcal N(x_0)$. Thus, $B+p=B+x_0-(1+t) u$ for some $t> 0$.

On the other hand, if $B+p=B+x-(1+t) u$ for some $x\in \partial P$, $u\in \mathcal N(x)$ and $t> 0$, then $B+x-u$ supports $P$ at $x$, and $(1+t)B+p$ also supports $P$ at $x$. This yields that $x$ is the farthest point of $P$ from $p$, and the uniqueness of $x_0$ and $u$ follows.
\end{proof}

Let $D(u,t)$ denote the disc-cap with normal $u$ and
height $t$. (Due to the strict convexity of $P$, $u$ determines $x_0$ uniquely.) Note that for each $u\in S^1$, there exists 
a maximal positive constant $t^*(u)$ such that $(B+x_u-(1+t)u)\cap P\neq\emptyset$
for all $t\in [0, t^*(u)]$.  Here $x_u$ is the unique point in $\partial P$ with $u\in \mathcal N(x_u)$. 
Let $A(u,t)=\area(D(u,t))$ and let $\ell(u,t)$ denote 
the arc-length of $\partial D(u,t)\cap (\partial B+x_u-(1+t)u)$.

We recall the following notations from \cite{FKV14}. 
Let $x$ and $y$ be two points from $P$. 
The two unit circles passing through $x$ and
$y$ determine two disc-caps of $P$, which we denote by
$D_-(x,y)$ and $D_+(x,y)$, respectively, such that $\area(D_-(x,y))\leq \area(D_+(x,y))$.
For brevity of notation, we write $A_-(x,y)=\area(D_-(x,y))$
and $A_+(x,y)=\area(D_+(x,y))$ and simply $A=\area (P)$. 

\begin{lemma}\label{vannagysapka}
Let $P$ be a disc-polygon with at least three vertices. Then there exists a constant $\delta_0>0$, depending only on $P$, such that
$A_+(x_1, x_2)> \delta_0$ for any two distinct points $x_1, x_2\in P$. 
\end{lemma}

\begin{proof} 
We note that $[x_1, x_2]_S$ cannot cover $P$ because $P$ is not a spindle. 
Thus, by compactness, there exists a constant $\delta_0>0$, depending only
on $P$, such that $\area(P\setminus [x_1,x_2]_S)> 2\delta_0$ for any two distinct points
$x_1, x_2\in P$. Now, the statement of the lemma follows from the fact that 
$P=D_-(x_1, x_2)\cup D_+(x_1, x_2)\cup [x_1, x_2]_S$.
\end{proof}

Note that the statement of Lemma~\ref{vannagysapka} does not hold if $P$ has only two vertices, that is, if it is a spindle $P=[v_1,v_2]_S$.
\begin{lemma}\label{2vertices}
Let $P=[v_1,v_2]_S$ be a disc-polygon with two vertices. Then there exists constants $c=c(P)$ and $\delta=\delta(P)$ such that if $x_1, x_2\in P$ with $A_-(x_1,x_2)\le A_+(x_1,x_2)<\delta$ then
$$|\overset{\LARGE\frown}{x_1x_2}|>c,$$
where $|\overset{\LARGE\frown}{x_1x_2}|$ denotes the arc-length of the shorter unit circular arc joining $x_1$ and $x_2$.
\end{lemma}

\begin{proof}
Similarly as before, 
\[
\area{[x_1,x_2]_S}\ge \area{[v_1,v_2]_S}-A_-(x_1,x_2)- A_+(x_1,x_2)>\area{[v_1,v_2]_S}-2\delta,
\]
and the assertion follows.
\end{proof}

\begin{lemma}\label{Vaztszerell}
Let $P$ be a disc-polygon and assume that the cap $D(u,t)$ is so small that $A(u,t)\le \delta$. Then 
$$\frac{ t \ell(u,t)}{2\pi}<A(u,t) <2 t \ell(u,t).$$ 
\end{lemma}

\begin{figure}[h]
    \centering
    \includegraphics[scale=.4]{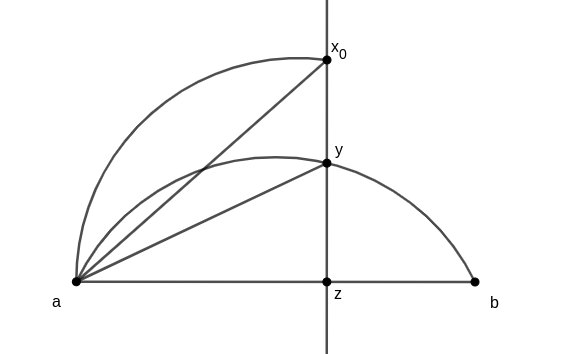}
    \caption{}
    \label{fig:bigcap}
\end{figure}

\begin{proof}
Let $x_0$ be the vertex of $D(u,t)=P\setminus (B^\circ+p)$, and assume that $\partial B+p$ intersects $\partial P$ in $a$ and $b$, consequently $\ell(u,t)>|ab|\ge 2 \ell(u,t)/\pi$. 

First we prove the lower bound. Draw a line $f$ through $x_0$ that is perpendicular to $ab$, let $z=f\cap ab$, and w.l.o.g. assume $|z-a|\ge |z-b|$. Denote by $y$ the intersection point of $f$ and $\partial B+p$, see Figure~\ref{fig:bigcap}. Note that $|y-x_0|\ge t$, and $|x_0-a|>|y-a|$.

Consider the domain $T$ bounded by the segment $x_0y$, the shorter circular arc joining $a$ and $x_0$, and the short circular arc joining $a$ and $y$, as on Figure~\ref{fig:bigcap}. Clearly, $T\subset D$. As $|x_0-a|>|y-a|$, it follows that the area of $T$ is larger then the area of the triangle $ayx_0$. Since $|az|\ge \ell(u,t)/\pi$ and $|x_0y|\ge t$, we have that $\area(ayx_0)\ge t\ell(u,t)/(2\pi)$, and the lower bound follows.

We turn to the upper bound. As the vertex is the farthest point of $P$ from $p$, it follows that $D$ is contained in an annulus of radii $1$ and $1+t$. Also, $D$ lies in the angle $apb\angle$. Hence
$$A(u,t)\le \frac{\ell(u,t)}{2\pi}\cdot ((1+t)^2-1)\pi<2t\ell(u,t),$$
which finishes the proof of the lemma.
\end{proof}

Assume that for a sufficiently small $t$ the cap $D(u,t)=P\setminus (B^\circ +p)$ contains a single vertex $v$ of $P$, and denote by $e$ end $e^*$ the two edges of $P$ that meet at $v$. Let $c$ be the centre of the unit circle that determines $e$, and $n=v-c$.  The circle $S^1+p$ intersects $S^1+c$ in $y$,  and the segment $pv$ in $z$, cf. Figure~\ref{fig:hatlestim}. Let $ \ell_1$ denote the shorter circular arc connecting $y$ and $z$, and let $\beta$ be the angle of $u$ and $n$. 

\begin{lemma} \label{lemma:1stasmyptotics}
With the notation above
\begin{equation}\label{estimatehatel}
\lim_{(t, \beta)\to (0+,0+)} \left ( \frac{\sin \ell_1 \cdot \sin \beta}{t}-\cos \ell_1 \right )= 0 
\end{equation}
\end{lemma}

\begin{figure}
    \centering
    \includegraphics[scale=.3]{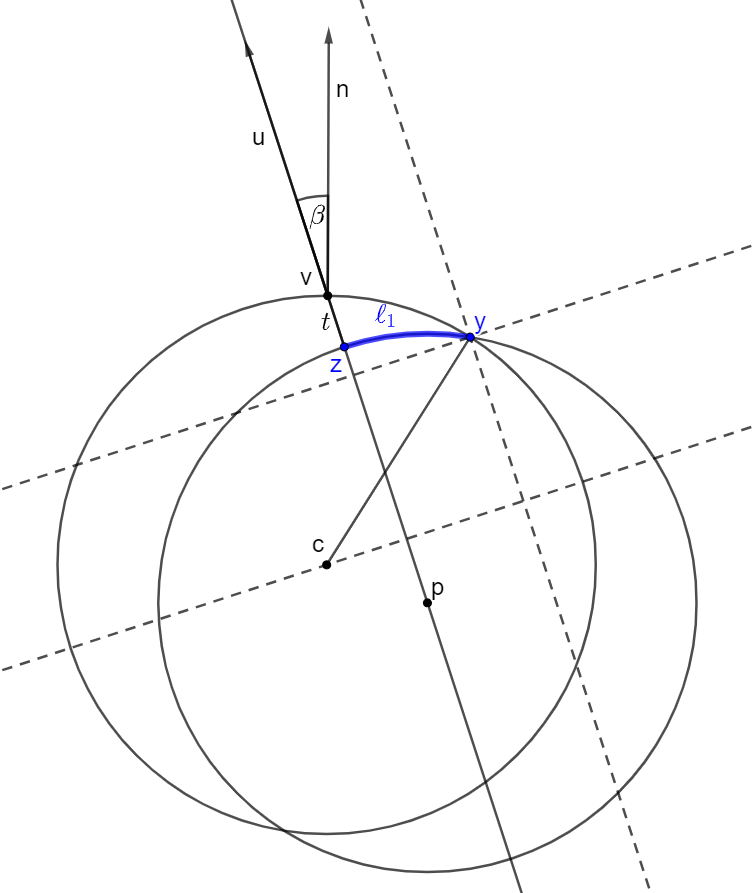}
    \caption{Computing $\ell_1$}
    \label{fig:hatlestim}
\end{figure}

\begin{proof}
We use the notations of Figure~\ref{fig:hatlestim}.

By the Pythagorean theorem $$|y-c|^2=1=(\sin \ell_1+\sin \beta)^2+(\cos \ell_1-(1+t-\cos \beta))^2.$$
After simplifying and rearranging the terms we get
$$\sin \ell_1 \sin \beta+ (\cos \ell_1 -1) (\cos \beta -1)=(\cos \beta+\cos \ell_1 -1)t -\frac {t^2}2. $$
Dividing by $t>0$ and using the $\sin ^2 x+\cos^2 x=1$ identity lead to
\begin{equation} \label{util1}\frac{\sin \ell_1 \cdot \sin \beta}{t}\left (1+ \frac{\sin \ell_1 \cdot \sin \beta}{(1+\cos \ell_1 )(1+\cos \beta)}\right )=\cos \beta +\cos \ell_1-1-\frac t2.
\end{equation}

As $\ell_1<\pi/2$ and we may clearly assume $\beta \le \pi/2$, the second factor on the left-hand-side is between $1$ and $2$, while the right-hand-side is  bounded. Hence $(\sin \ell_1 \sin \beta)/t$ is also bounded, and thus 
$$\lim_{(t, \beta)\to (0+,0+)} \sin \ell_1 \cdot \sin \beta=0,$$
which implies \eqref{estimatehatel} using (\ref{util1}).
\end{proof}

Keeping $\beta > 0$ fixed, from \eqref{util1} we obtain
\begin{equation} \label{eq:l-asy-betanagy}
\ell_1(\beta, t) \sim t \cot{\beta}, \quad \text{as } t \to 0+.
\end{equation}

Let $A_1(\beta, t)$ denote the area of the set bounded by 
the arcs $vy$ and $yz$, and the segment $vz$, see Figure \ref{fig:hatlestim}.

\begin{lemma} \label{lemma:A-area}
For any $\varepsilon > 0$ there exists $\delta > 0$ such that 
if $t \leq \delta \beta$ and $\beta < \delta$, then
\[
\frac{1 - \varepsilon}{2} t \ell_1(\beta, t) \leq A_1 (\beta, t) 
\leq \frac{1 + \varepsilon}{2} t \ell_1(\beta, t).
\]
\end{lemma}

\begin{proof}
Let $i$ denote the length of the arc $\overset{\LARGE\frown}{vy}$, and put $f(x) = x - \sin x$. 
Then $f(i)/2$ is the area of the set between the arc $\overset{\LARGE\frown}{vy}$ and the segment $vy$.
Therefore,
\begin{equation} \label{eq:A1-ter}
A_1(\beta, t) = \area(yvz) + \frac{1}{2} (f(i) - f(\ell)).
\end{equation}

We claim that 
\[
f( i ) - f(\ell_1) \leq \varepsilon t \ell_1.
\]
By the triangle inequality in $yvz$ we obtain
\[
2 \sin \frac{i}{2} - 2 \sin \frac{\ell_1}{2} \leq t.
\]
We have
\[
t \geq 
2 \sin \frac{i}{2} - 2 \sin \frac{\ell_1}{2} = ( i - \ell_1) \cos \xi 
\geq \frac{i - \ell_1}{2},
\]
where $\xi \in (\ell_1, i)$. Thus, 
$i - \ell_1 \leq 2t$. Furthermore,
\begin{equation} \label{eq:f-bound}
 f(i) - f(\ell_1) = (i - \ell_1) f'(\xi')
 \leq (i - \ell_1) \frac{i^2}{2} \leq i^2 t
 \leq 4 \ell_1^2 t,
\end{equation}
where, in the last inequality we used that $i \leq 2 \ell_1$. Since $\ell_1$ is small,
for small enough $\delta > 0$
\[
( 1 - \varepsilon / 2) \frac{t \ell_1(\beta, t)}{2} \leq 
\area(yvz) \leq ( 1 + \varepsilon / 2) \frac{t \ell_1(\beta, t)}{2},
\]
thus the result follows from \eqref{eq:A1-ter} and \eqref{eq:f-bound}.
\end{proof}

The following simple corollary adds to Lemma~\ref{lemma:1stasmyptotics}.

\begin{cor}\label{uniformnagy}
Assume that the cap $D(u,t)$ contains at least two vertices of $P$ and that $A(u,t)\le \delta_0$. Then there is a constant $c>0$ (depending only on $P$) such that for all possible $t>0$ and $u\in S^1$ we have
$$c<\ell(u,t) \qquad { and } \qquad ct<A(u,t).$$ 
\end{cor}

\begin{proof}
Let $c_0=\min |pq|$ where $p$ and $q$ are two points from $\partial P$ that are not on adjacent edges of $P$. Obviously, $c_0<\ell(u,t)$. The second part of the statement follows from Lemma~\ref{Vaztszerell}.
\end{proof}

\section{Proof of Theorem \ref{main}}

We only prove \eqref{main:vert} concerning the vertex (or edge) number. The asymptotic formula \eqref{main:area} for the missed area follows directly from the Efron-type identity (5.10) in \cite{FKV14}.

First we assume that $P$ has at least $3$ vertices. Observe that the pair of random points $x_1, x_2$ determines an edge of $P_n$
if and only if at least one of the disc-caps 
$D_-(x_1,x_2)$ and $D_+(x_1, x_2)$ 
does not contain any other points from $X_n$. Thus, using the notation from \cite{RS63},
$$
\EE (f_0(P_n))= \binom{n}{2}  W_n,
$$
where
\begin{equation}\label{dabjuen}
W_n=\frac1{A^2}\int_P \int_P
\left[ \left (1-\frac{A_-(x_1, x_2)}{A} \right )^{n-2} + \left (1-\frac{A_+(x_1, x_2)}{A}
\right )^{n-2} \right] \dx x_1 \dx x_2.
\end{equation}
Note that if all points of $X_n$ fall into the closed spindle spanned by $x_1$ and $x_2$,
then $x_1$ and $x_2$ contribute two edges to $P_n$
(since in this case $[X_n]_S=[x_1,x_2]_S$), 
and accordingly, this event is counted in both terms in the integrand of (\ref{dabjuen}).

As $f_0(P)\ge 3$ is assumed, Lemma~\ref{vannagysapka} yields that 
\begin{align*}
&\lim_{n\to\infty} 
\binom{n}{2}\frac{1}{A^2}\int_P\int_P \left (1-\frac{A_+(x_1, x_2)}{A}\right )^{n-2} \dx x_1 \dx x_2\\
\leq&\lim_{n\to\infty} \binom{n}{2}\frac{1}{A^2}\int_P\int_P e^{-\frac{\delta_0}{A}(n-2)} \dx x_1 \dx x_2\\
=&\lim_{n\to\infty} \binom{n}{2} e^{-\frac{\delta_0}{A}(n-2)}=0.
\end{align*}
Thus, the contribution of the second term of (\ref{dabjuen}) 
is negligible, hence, in what follows, we will consider only the first term.
Note that a similar argument yields that in the first term of (\ref{dabjuen})
it is enough to integrate over pairs of random points $x_1, x_2$ such that
$A_-(x_1, x_2)<\delta_0$. Furthermore, the same conclusion holds for any fixed $\delta\leq \delta_0$. Let ${\mathbf 1}(\cdot)$ denote the indicator function of 
an event. Then
\begin{multline}\label{eq:restricted}
\lim_{n\to\infty} \EE(f_0(P_n))\frac 1 {\ln n}\\
=\lim_{n\to\infty} \frac 1 {\ln n} \binom{n}{2}
\frac{1}{A^2}\int_P\int_P \left (1-\frac{A_-(x_1, x_2)}{A}\right )^{n-2}
{\bf 1}(A_-(x_1, x_2)<\delta_0) \dx x_1 \dx x_2. 
\end{multline}

Now, we re-parametrise the pair $(x_1, x_2)$ as follows, see \cite{FKV14} and \cite{Santalo}. Let 
\begin{equation}\label{kulcstrafo}
(x_1, x_2)=\Phi (u, t, u_1, u_2),
\end{equation}
where $u, u_1, u_2\in S^1$ and $0\leq t\leq t_0(u)$ are chosen such that $\area D(u,t)< \delta_0$, and thus 
$$D(u,t)=D_-(x_1, x_2),$$ 
and
$$(x_1, x_2)=(x_u-(1+t)u+u_1, x_u-(1+t)u+u_2).
$$
Note that $u_1$ and $u_2$ are the unique outer unit normal vectors
of $\partial B+x_u-(1+t)u$ at $x_1$ and $x_2$, respectively. 
This yields that, for fixed $u$ and $t$, both $u_1$ and $u_2$
are in the same arc of length $\ell (u,t)$ in $S^1$. We denote
this arc by $L(u,t)$. 
Since $A_-(x_1, x_2)< \delta_0$,  
$D_-(x_1, x_2)$ is uniquely determined by Lemma~\ref{vannagysapka}.
Now, the uniqueness of the vertex and height of a disc-cap guarantees 
that $\Phi$ is well-defined, bijective, and differentiable 
on a suitable domain of $(u, t, u_1, u_2)$, cf. \cite{FKV14}.

Let $v_0, \ldots, v_{k-1}$ denote the vertices of $P$ labelled cyclically on $\partial P$ in the positive direction, and let 
$\mathcal{N}(v_i) = \overset{\LARGE\frown}{n_im_i} \subset S^1$, which is a closed arc of $S^1$. Let $r: [0,2\pi)\to S^1$ be the usual parametrisation of the unit circle, and we introduce $\alpha_i=r^{-1}(n_i)$, $\beta_i=r^{-1}(m_i)$, for an arbitrary $u\in S^{1}$ we use $\beta=r^{-1}(u)$, and for simplicity we write $D(\beta, t)=D(r(\beta), t)$, etc. accordingly.
Put 
\[
N_1 = \cup_{i=0}^{k-1} \mathcal{N}(v_i)  \subset S^1,
\quad N_2 = S^1 \backslash N_1,
\]
and 
\[
B_1 = \{ (x_1, x_2) \in P^2: \,  u \in N_1, \text{ where } D(u,t)= D(x_1, x_2) \},
\quad B_2 = P^2 \backslash B_1.
\]
The same calculation as in the Appendix of \cite{FKV14} yields that 
the Jacobian $|J\Phi|$ of $\Phi$ satisfies
\begin{equation}\label{jacobi}
|J\Phi(u,t,u_1, u_2)|=
\begin{cases} 
( 1+t )|u_1\times u_2|, & \text{if } u \in N_1, \\
t |u_1\times u_2|, & \text{if } u \in N_2.
\end{cases}
\end{equation}
We note that $|u_1\times u_2|$ equals the sine of the 
length of the unit 
circular arc between $x_1$ and $x_2$ on the boundary of $D(u,t)$.

Notice that if $u \in N_2$, then for any $t > 0$ the cap $D(u,t)$ contains 
two (or more, if $t$ is large) vertices of $P$, while if $u \in N_1$, then for 
sufficiently small $t$ (depending on $u$), the cap $D(u,t)$ contains 
exactly one vertex of $P$.

First, we show that the part of the integral in \eqref{eq:restricted} on $B_2$ is negligible. Note that the height $t_0(u)$ is uniformly bounded: $0<t_1\le t_0(u) \le t_2$. 
Using \eqref{jacobi} and Corollary \ref{uniformnagy}
we have
\[
\begin{split}
& \int \!\! \int_{B_2} 
\left (1-\frac{A_-(x_1, x_2)}{A}\right )^{n-2}
{\bf 1}(A_-(x_1, x_2)<\delta) \dx x_1 \dx x_2  \\
& = \int_{N_2} \int_0^{t_0(u)} \int_{L(u,t)} \int_{L(u,t)}
\left( 1 - \frac{A(u,t)}{A} \right)^{n-2} 
t  |u_1\times u_2| \, \dx u_1 \dx u_2 \,
\dd t \dd u \\
& = \int_{N_2} \dd u \int_0^{t_0(u)} 
\left( 1 - \frac{A(u,t)}{A} \right)^{n-2} t ( \ell(u,t) - \sin \ell(u,t) ) \dd t \\
& \leq C \int_0^{t_2} ( 1 - ct)^{n-2} t \dd t = O(n^{-2}),
\end{split}
\]
where, here and later on, $c, C$ are strictly positive generic constant, whose exact 
value is not important and can be different at each appearance.
We also used that $t_2$ can be chosen sufficiently small to guarantee that $ct_2<1$. Furthermore, the variables $u_1$ and $u_2$ appear only in the 
$|u_1 \times  u_2 |$ term, thus  the inner double integral can be evaluated 
explicitly.
In summary, the integral on $B_2$ is negligible.

\smallskip

Next we deal with the part of the integral in \eqref{eq:restricted} on $B_1$. We have that
\[
\begin{split}
& \int \!\! \int_{B_1} 
\left (1-\frac{A_-(x_1, x_2)}{A}\right )^{n-2}
{\bf 1}(A_-(x_1, x_2)<\delta) \dx x_1 \dx x_2  \\
& = \int_{N_1} \dd u \int_0^{t_0(u)} 
\left( 1 - \frac{A(u,t)}{A} \right)^{n-2} (1+t) ( \ell(u,t) - \sin \ell(u,t) ) \dd t.
\end{split}
\]
Replacing $t_0(u)$ with $t_1$ we lose a negligible part of the integral, as before.

We split the integral further according to the vertices. Fix $\varepsilon > 0$
small enough. 
If $\beta \in [\alpha_i + \varepsilon, \beta_i - \varepsilon]$, 
$i = 0, 1 \ldots, k-1$, then by 
\eqref{eq:l-asy-betanagy} it follows that
\[
 \ell(\beta,t) \sim t ( \cot (\beta - \alpha_i) + \cot (\beta_i - \beta)).
\]
Thus, by Lemma \ref{Vaztszerell},
$A(\beta,t) \geq c t^2$ 
uniformly in $\beta \in [\alpha_i + \varepsilon, \beta_i - \varepsilon]$.
Therefore, for a fixed $\varepsilon>0$, for each $i = 0,1,\ldots, k-1$, it holds that 
\begin{align}
\nonumber & \int_{\alpha_i + \varepsilon}^{\beta_i - \varepsilon} \dd \beta 
\int_0^{t_1} \left( 1 - \frac{A(\beta,t)}{A} \right)^{n-2} (1+t) 
( \ell(\beta,t) - \sin \ell(\beta,t) ) \dd t \\
\nonumber & \leq C \int_0^{t_1} ( 1 - c t^2)^n t^3 \dd t = O(n^{-2}).
\end{align}
Therefore, the main contribution of the integral \eqref{eq:restricted} comes from the corners.

\smallskip

For simplicity, choose the vertex $v_0$ and assume that $\alpha_0 = 0$.
We determine the contribution of the integral on $\beta \in (0,\varepsilon)$. 
Introduce the notation
$$I=\int_0^\varepsilon \dx \beta \int_0^{t_1} \left (1- \frac{A(\beta,t)}{A} \right)^{n-2}(1+t)(\ell(\beta,t)-\sin \ell(\beta,t)) \dx t.$$
Let $\delta > 0$ be a fixed small number, to be determined later.
We split $I$ as follows
\begin{align} \label{eq:I12}
I_1 & = 
\int_0^\varepsilon \dd \beta \int_{\delta \beta}^{t_1}
\left( 1 - \frac{A(\beta, t)}{A} \right)^{n-2} (1+t) (\ell - \sin \ell) \, \dd t ,\\
I_2 & = 
\int_0^\varepsilon \dd \beta \int_0^{\delta \beta}
\left( 1 - \frac{A(\beta, t)}{A} \right)^{n-2} (1+t) (\ell - \sin \ell) \, \dd t.
\end{align}
First we show that  $I_1$ is negligible for any $\delta > 0$ and $\varepsilon > 0$.

To simplify notation, 
put $\ell_1=\ell_1(\beta, t)$ and $\ell_2=\ell-\ell_1$ (as in Lemma \ref{lemma:1stasmyptotics}), 
and let $A_i$ be the area corresponding to $\ell_i$, $i=1,2$ (see Figure~\ref{fig:A1A2}).

We note that $\ell_2$ is small, since it follows from \eqref{eq:l-asy-betanagy} that
\begin{equation} \label{eq:ell2}
\ell_2(\beta, t) \sim t \cot (\beta_0 - \beta) \quad \text{as }  t \to 0+,
\end{equation}
uniformly in $\beta \leq \varepsilon$.

\begin{figure}[h]
    \centering
    \includegraphics[scale=.3]{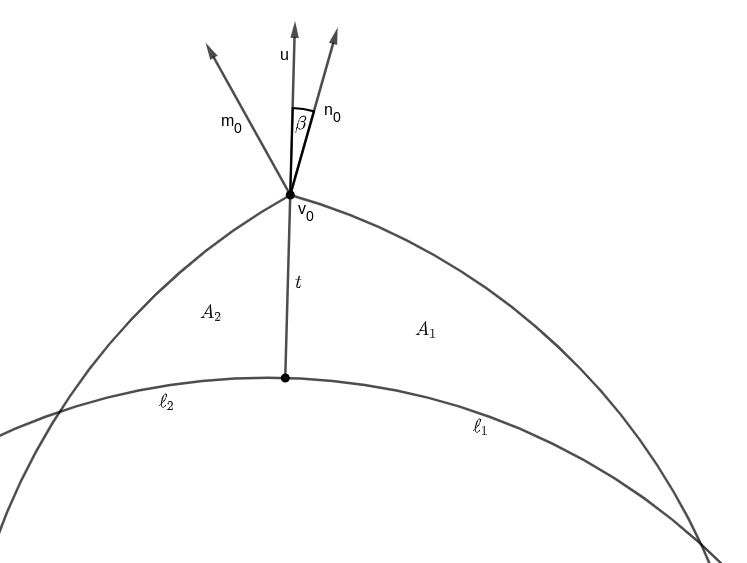}
    \caption{}
    \label{fig:A1A2}
\end{figure}

Now assume that $t > \delta \beta$ and $\ell_1<\delta/2$. Then
\[
\frac{\sin \ell_1 \sin \beta}{t} <\frac{\frac{\delta}{2} \beta}{\delta \beta}=\frac 12,
\]
which contradicts Lemma~\ref{lemma:1stasmyptotics} if $\delta$ is sufficiently small. 
Therefore
\begin{equation} \label{eq:l-nagy}
\ell_1(\beta, t) \geq \delta / 2, \quad \text{if } t \geq \delta \beta.
\end{equation}

Also, if $t \geq \delta\beta$, then by Lemma \ref{Vaztszerell} and \eqref{eq:l-nagy}
\begin{equation} \label{eq:A1-asy2}
ct\leq A_1(\beta, t). 
\end{equation}
Similarly, Lemma \ref{Vaztszerell} and \eqref{eq:ell2} yield that 
\begin{equation} \label{eq:A2-asy}
c t^2 \leq A_2(\beta, t) \leq C t^2.
\end{equation}
Now, by \eqref{eq:A1-asy2} and \eqref{eq:A2-asy} we obtain that
\[
\begin{split}
I_1 & \leq 2 \int_0^\varepsilon \dd \beta \int_{\delta \beta}^{t_1}
\left( 1 - \frac{A(\beta, t)}{A} \right)^{n-2} \dd t \\
& \leq 2 \int_0^{t_1} \frac{t}{\delta} \left( 1 - c t 
\right)^{n-2} \dd t = O(n^{-2}),
\end{split}
\]
 which proves that $I_1$ is negligible.

Finally, we estimate $I_2$, which carries the weight of the integral in \eqref{dabjuen}. Let $\varepsilon_1 > 0$ be fixed. We apply \eqref{eq:ell2} and Lemmas \ref{lemma:1stasmyptotics} and \ref{lemma:A-area},  and  we choose $\delta > 0$ and $\varepsilon > 0$ 
small enough such that 
\[
\begin{split}
& (1 - \varepsilon_1) \frac{t^3}{6} \leq t - \sin t \leq 
(1 + \varepsilon_1) \frac{t^3}{6}, \quad t \in [0,\delta], \\
& (1 - \varepsilon_1) \frac{t}{\beta} \leq \ell(\beta, t) \leq 
(1 + \varepsilon_1) \frac{t}{\beta}, \quad t/\delta \leq  \beta \leq  \varepsilon \\
& (1 - \varepsilon_1) \frac{t^2}{2 \beta} \leq A(\beta, t) \leq 
(1 + \varepsilon_1) \frac{t^2}{2 \beta}, \quad t/\delta \leq  \beta \leq  \varepsilon.
\end{split}
\]
Substituting $y = t^2 (1-\varepsilon_1) / ( 2A\beta) =: d_1 t^2/\beta$ and
changing the order of integration yield
\[
\begin{split}
I_2 & \leq \int_0^\varepsilon \dd \beta \int_0^{\delta \beta} 
\left( 1 - \frac{(1- \varepsilon_1) t^2}{2 \beta A} \right)^{n-2}
\frac{t^3 (1+\varepsilon_1)^3}{6 \beta^3} ( 1 + \delta \varepsilon) \dd t \\
& = \frac{( 1 + \delta \varepsilon)(1 + \varepsilon_1)^3}{12 \, d_1^2} 
\int_0^\varepsilon \dd \beta 
\int_0^{\delta^2 d_1 \beta} ( 1- y)^{n-2} \frac{y}{\beta} \dd y \\
& = 
\frac{( 1 + \delta \varepsilon) (1 + \varepsilon_1)^3}{12 \, d_1^2 \,  n^2}
\int_0^{n\delta^2 \varepsilon d_1} \left( 1 - \frac{x}{n} \right)^{n-2}
x \left( \ln (\varepsilon \delta^2 d_1 n) - \ln x \right) \dd x \\
& \sim 
\frac{\ln n \, ( 1 + \delta \varepsilon)(1 + \varepsilon_1)^3}{12 \, d_1^2 \, n^2}\quad  \text{as } n\to\infty. 
\end{split}
\]
Since $\varepsilon_1 > 0$ is arbitrary, and the lower bound can be obtained by an analogous argument, we have obtained that
\begin{equation} \label{eq:I2-asy}
I_2 \sim 
\frac{A^2 \ln n }{3  n^2}  \quad \text{as } n\to \infty.
\end{equation}
Since at each vertex we have twice the contribution of $I_2$, the statement follows
when $f_0(P)\ge 3$.

To finish the proof we need to deal with the case in which $f_0(P)=2$. By Lemma~\ref{2vertices}, if both $A_-(x_1,x_2)$ and $A_+(x_1,x_2)$ are small, then $\ell$ is larger than an absolute constant, and this part of the integral can be estimated similarly to $I_1$. The rest of the argument remains valid in this case as well.
 
\section*{Acknowledgements}
F. Fodor was partially supported by Hungarian NKFIH grant K134814. 
P. Kevei was partially supported by the J\'{a}nos Bolyai Research 
Scholarship of the Hungarian Academy of Sciences.
V. V\'\i gh was partially supported by  Hungarian NKFIH grant FK135392. 
This research was supported by grant NKFIH-1279-2/2020 of the Ministry for Innovation and Technology, Hungary.

\end{document}